\documentclass[11pt]{amsart}
\usepackage{graphicx}
\usepackage{oldgerm}
\usepackage{mathdots}
\usepackage{stmaryrd}
\usepackage{bm}
\usepackage[all]{xy}
\usepackage{color}
\usepackage{enumerate}

\usepackage{hyperref}
\usepackage{mathrsfs}
\usepackage{tikz}
\usetikzlibrary{arrows,%
  decorations.markings,%
  matrix,%
  shapes}

\usepackage[latin1]{inputenc}
\usepackage{amsmath,amsthm, amscd, amssymb, amsfonts}

\numberwithin{equation}{section}

\theoremstyle{plain}

\numberwithin{equation}{section}

\newtheorem{theorem}[equation]{Theorem}
\newtheorem{corollary}[equation]{Corollary}
\newtheorem{proposition}[equation]{Proposition}
\newtheorem{lemma}[equation]{Lemma}

\theoremstyle{definition}
\newtheorem{definition}[equation]{Definition}

\newtheorem{example}[equation]{Example}

\theoremstyle{definition}
\newtheorem{remark}[equation]{Remark}

\newtheorem*{assumption}{Assumption}

\newcommand{\A}{\mathscr{A}}
\newcommand{\C}{\mathscr{C}}

\newcommand{\M}{{\mathfrak{m}}}
\newcommand{\N}{{\mathbb N}}

\newcommand\id{\operatorname{id}}

\newcommand\Ext{\operatorname{Ext}}

\newcommand\Homol{\operatorname{H}}
\newcommand\coh{\Homol}

\newcommand\cx{\operatorname{cx}}

\newcommand\ot{\otimes}
\newcommand\Hom{\operatorname{Hom}}

\newcommand\FPdim{\operatorname{FPdim}}

\newcommand\VC{V_{\C}}
\newcommand\unit{\mathbf{1}}
\DeclareMathOperator{\Ima}{Im}
\DeclareMathOperator{\Ker}{Ker}
\newcommand{\DOT}{\setlength{\unitlength}{1pt}\begin{picture}(2.5,2)(1,1)\put(2,3){\circle*{2}}\end{picture}}
\newcommand{\bu}{\DOT}

\newcommand{\Coh}{\operatorname{H}\nolimits}
\newcommand{\Ho}{\operatorname{\Coh^{\bu}}\nolimits}
\newcommand{\Maxspec}{\operatorname{MaxSpec}\nolimits}

\newcommand{\az}{\mathfrak{a}}

\newcommand{\m}{\mathfrak{m}}

\makeatletter
\def\blx@maxline{77}
\makeatother

\begin{document}
\title[Support varieties for finite tensor categories]
{Support varieties for finite tensor categories:
Complexity, realization, and connectedness}

\author{Petter Andreas Bergh, Julia Yael Plavnik, Sarah Witherspoon}

\address{Petter Andreas Bergh \\ Institutt for matematiske fag \\
NTNU \\ N-7491 Trondheim \\ Norway} \email{petter.bergh@ntnu.no}
\address{Julia Yael Plavnik \\ Department of Mathematics \\ Indiana University \\ Bloomington \\ Indiana 47405 \\ USA}
\email{jplavnik@iu.edu}
\address{Sarah Witherspoon \\ Department of Mathematics \\ Texas A \& M University \\ College Station \\ Texas 77843 \\ USA}
\email{sjw@math.tamu.edu}
\thanks{The second author was partially supported by NSF grants 1802503 and 1917319.
The third author was partially supported by NSF grant 1665286.
}
\date{\today}

\begin{abstract}
We advance support variety theory for finite tensor categories. First we show that the dimension of the support variety of an object
equals the rate of growth of a minimal projective resolution as measured by the Frobenius-Perron dimension. Then we show that every conical subvariety of the support variety of the unit object may be realized as the support variety of an object. Finally, we show that the support variety of an indecomposable object is connected. 
\end{abstract}

\maketitle

\section{Introduction}

\sloppy Tensor categories arise in many important settings such as
representation theory, low dimensional topology, and quantum computing.
Nonsemisimple tensor categories range from categories of representations
of finite groups in positive characteristic
and representations of some finite dimensional Hopf algebras 
to categories discovered more recently such as 
those appearing in logarithmic conformal field theory~\cite{EO,Gaberdiel},
representations of dynamical quantum groups at roots of unity~\cite{NP}, 
and some new categories in characteristic two~\cite{BE}.
For nonsemisimple tensor categories satisfying some finiteness conditions, 
support varieties are meaningful geometric invariants of objects.  
Their theory began in work of Quillen~\cite{Quillen} and Carlson~\cite{Carlson1}
on finite group representations. 
In more recent years, the theory of support varieties was generalized
in many directions, for example to representations of 
Hopf algebras~\cite{BKN,Feldvoss-W,FriedlanderPevtsova,Ostrik,PlW1,Witherspoon}
and of finite dimensional self-injective algebras~\cite{EHSST,SS04} 
and to objects in triangulated categories~\cite{BKSS}.

In this paper, we advance support variety theory for 
finite tensor categories generally, with a view toward further applications. For every finite tensor category $\C$, the cohomology ring $\coh^*(\C)=\Ext^*_{\C}(\unit,\unit)$, where $\unit$ is the unit object, is a graded commutative ring. As a consequence, the support variety $V_{\C}(X)$ of an object 
$X$ -- defined in terms of the annihilator of $\Ext^*_{\C}(X,X)$ in the cohomology ring $\coh^*(\C)$ --
is a topological space in the Zariski topology. Many properties of support varieties hold in this full generality, without any further assumptions. However, in order to develop a robust support variety theory, it is necessary to require a finiteness condition
that is known to hold in many cases: 
Etingof and Ostrik~\cite{EO} conjectured 
that the cohomology ring $\coh^*(\C)$ is finitely generated, and that $\Ext^*_{\C}(X,X)$ is a finitely generated $\coh^*(\C)$-module for all objects $X \in \C$. When this holds, the support varieties encode homological properties of the objects. The main results in this paper are examples of this.

We make significant contributions in a few useful directions.
We define the complexity of an object $X$ as the rate of growth of 
a minimal projective 
resolution of $X$ as measured by the Frobenius-Perron dimensions
of its components. 
We show in Theorem~\ref{thm:cx-dim} that just as for finite dimensional
Hopf algebras~\cite{Feldvoss-W}, the complexity of $X$ is equal to the 
dimension of its support variety, $\dim V_{\C}(X)$.
We then recall a standard construction
of some special modules first defined for finite groups
by Carlson, and apply it to objects in tensor categories 
(cf.~the Koszul objects in~\cite{BKSS}), to define
objects $L_{\zeta}$ in Section~\ref{sec:carlson}.
We show that they satisfy a 
tensor product property, 
that is the support variety of a tensor product of an object with
any $L_{\zeta}$ 
is the intersection of their support varieties. 
We use them to show that any conical subvariety of the support variety of the
unit object may be realized as the support variety of some object. 
These objects $L_{\zeta}$ also play a key role in 
generalizing a result of Carlson~\cite{Carlson} from 
indecomposable modules for a finite group to indecomposable 
objects in a finite tensor category $\C$. 
Namely, we show in Theorem~\ref{thm:main} and Corollary~\ref{cor:connected} that  
the support variety of an indecomposable object is connected.
The proof requires Proposition~\ref{prop:reducing} that allows
us to reduce the complexity of an object for use in inductive arguments.
We also give a needed connection between the vanishing of Ext 
and dimensions of varieties in Proposition~\ref{prop:vanishing}.

As a word of caution, we observe that some standard properties of 
support varieties for finite groups do not always hold in this 
general setting of finite tensor categories.
For example, the varieties of an object and of its dual need not be the same,
and the variety of a tensor product of objects need not be
the intersection of their varieties.
See, e.g.,~\cite{BW,PlW1} for counterexamples. 
These counterexamples occur in categories that are not braided;
we do not know whether these statements always hold in a braided category.
They are however used in some proofs that the variety
of an indecomposable module is connected
(see, e.g.,~\cite{Benson2,EHSST}).
We take the alternative route as 
outlined above to demonstrate
connectedness of the variety of an indecomposable object
without relying on these properties.

The contents of this paper are as follows. 
In Section~\ref{sec:prelim}, we recall the definitions of finite
tensor categories, the Frobenius-Perron dimension,
projective covers and minimal resolutions, and state some needed lemmas.
In Section~\ref{supp-var}, we define support varieties for objects
of a finite tensor category $\C$, 
and conclude some standard properties.
We then state the finite generation condition on the cohomology of $\C$, and this condition will be assumed in most of the results in the rest of the paper. 
In Section~\ref{sec:complexity}, we define the complexity of an object and
show that it agrees with the dimension of the support variety.
As a consequence, we show that an object is projective
if and only if its support variety is zero-dimensional.
In Section~\ref{sec:carlson}, for each homogeneous positive
degree element $\zeta$ of the cohomology ring of the 
finite tensor category $\C$,
we define an object $L_{\zeta}$ whose variety is the zero set
of the ideal generated by $\zeta$.
We obtain, as a standard
consequence of the definition of $L_{\zeta}$, both a tensor product property
and a realization result: any conical subvariety of
the support variety of the unit object $\unit$ can be realized
as the support variety of some object.
Finally, in Section~\ref{sec:indecomposable}, we show that 
the variety of an indecomposable object is connected.

%%%%%%%%%%%%%%%%%%%%%%%%%%%%%

\subsection*{Acknowledgments}
We thank Dave Benson, Karin Erdmann, Henning Krause, Mart\'in Mombelli, Cris Negron, 
and Victor Ostrik for very helpful conversations that
led to improvements in this paper.

\section{Preliminaries}\label{sec:prelim}

In this section, we summarize some basic facts about projective covers
and resolutions in a finite tensor category $\C$, and recall the definition of the Frobenius-Perron dimension. For details, we refer the reader to \cite{BK, EGNO, EO, NP}. Throughout, we fix an algebraically closed field $k$ of arbitrary characteristic.

Recall that a {\em finite tensor category} $\C$ is a locally finite
$k$-linear abelian category with finitely
many simple objects (up to isomorphism)
and enough projectives
together with
a bifunctor $\ot: \C\times \C \rightarrow \C$ that is associative
(up to functorial isomorphisms), bilinear on morphisms, and
satisfies some associativity axioms.
In addition, there is a unit object $\unit$ in $\C$ (an identity with
respect to $\ot$ up to functorial isomorphism) that is simple, and
every object in $\C$ has both left and right duals, i.e.\ $\C$ is rigid. 
This requirement is important for us, even though we shall not be using dual objects directly. For example, by \cite[Remark 6.1.4]{EGNO}, it implies that the category is quasi-Frobenius, that is, the projective objects and the injective objects are the same. 
Note that since the underlying category $\C$ is locally finite, the Jordan-H{\"o}lder Theorem and the Krull-Schmidt Theorem hold; see \cite[Section 1.5]{EGNO}.
Recall that
$k$-linear means that the
morphism sets are $k$-vector spaces for which composition of morphisms
is $k$-bilinear.

From now on, $\C$ will be a finite tensor category. It follows that $\C$ is equivalent to the category of finite dimensional
modules over some finite dimensional $k$-algebra~\cite[p.~9]{EGNO}. 

Some additional properties of the tensor product $\ot$ ensured by rigidity are:
\begin{itemize}
\item[(i)] The tensor product $\ot$ is 
biexact~\cite[Proposition 4.2.1]{EGNO}.
\item[(ii)] Whenever $X$ is an object
and $P$ is a projective object of $\C$, the objects 
$P\ot X$ and $X\ot P$ are also projective~\cite[Proposition 4.2.12]{EGNO}.
\end{itemize} 
We will sometimes take $\C$ to be {\em braided},
meaning that there are functorial isomorphisms 
$X\ot Y\cong Y\ot X$ for all objects $X,Y$ in $\C$
that satisfy some hexagonal identities~\cite[Definition 8.1.1]{EGNO}.

\subsection*{Projective covers and stable isomorphisms}
Let $X$ be an object in the finite tensor category $\C$. 
A {\em projective cover} of $X$ is
a projective object $P(X)$ in $\C$ together with an epimorphism $p: P(X)
\rightarrow X$ such that if $f:P\rightarrow X$ is an epimorphism
from a projective object $P$ to $X$, then there is an epimorphism
$g:P\rightarrow P(X)$ for which $pg=f$~\cite[Definition~1.6.6]{EGNO}.
Projective covers exist and 
are unique up to nonunique isomorphism~\cite[p.~6]{EGNO}.

Let $X_1,\ldots, X_r$ be the simple
objects in $\C$ (one from each isomorphism class). 
We will use the following equation involving vector space
dimensions of morphism spaces:
for any object $Y$ in $\C$, denote by $[Y:X_i]$ the multiplicity
of the simple object 
$X_i$ in a Jordan-H\"older series of $Y$. By~\cite[Equation (1.7)]{EGNO},
\begin{equation*}\label{eqn:dim-Hom}
  [Y:X_i] =
       \dim_k \Hom_{\C} (P(X_i),Y) . \tag{$\dagger$}
\end{equation*}
We will need the following presentation of a projective object. 

\begin{lemma}\label{lem:proj}
Let $P$ be a projective object in $\C$.
Then
$
   P\cong \oplus _{i=1}^r a_i P(X_i)
$
for some nonnegative integers $a_i$,
where $P(X_i)$ is the projective cover of the simple object $X_i$
for each $i$.
Moreover, $a_i= \dim_k\Hom_{\C}(P,X_i)$.
\end{lemma}

\begin{proof}
This follows from the existence of a category equivalence with the category of modules
over some finite dimensional algebra (see~\cite[pp.~9--10]{EGNO})
and standard facts about finite dimensional algebras.
However, we give a more direct proof in our setting, by induction on
the length $n$ of a Jordan-H\"older series for the 
projective object $P$.

If $n=1$, then $P$ is both simple and projective, and so 
$P = X_i= P(X_i)$ for some $i$.
Now assume that the first statement holds for all projective
objects of length less than $n$. 
For some $i$, the vector space $\Hom_{\C}(P,X_i)$ is nonzero. 
Choose a nonzero morphism
$f$ in $\Hom_{\C}(P,X_i)$.
Since $X_i$ is simple, $f$ is an epimorphism.
By definition of a projective cover $p: P(X_i)\rightarrow X_i$, 
there exists an epimorphism $g:P\rightarrow P(X_i)$ such that
$p g = f$. 
Since $g$ is an epimorphism and $P(X_i)$ is projective,
there is a splitting morphism $h:P(X_i)\rightarrow P$, that is $gh = \id_{P(X_i)}$,
the identity morphism on $P(X_i)$. 
It follows that $P(X_i)$ is a direct summand of $P$.
Write $P\cong P(X_i)\oplus Q$ for some projective object $Q$.
Then $Q$ has length less than $n$, and by the
induction hypothesis, it has a direct sum decomposition
as in the first statement of the theorem.
It follows that $P$ does as well. 

Now write $P\cong \oplus_{j=1}^r a_j P(X_j)$.
By equation~(\ref{eqn:dim-Hom}), for each $i$,
\[
  \dim_k \Hom_{\C} (P,X_i) = 
     \sum_{j=1}^r a_j \dim_k \Hom_{\C}(P(X_j), X_i) 
       = \sum_{j=1}^r a_j [X_i:X_j] = a_i .
\]
\end{proof}

The next result is Schanuel's Lemma for abelian categories; we include a proof for completeness. Let us call two objects $X,Y$ \emph{stably isomorphic} if there exist projective objects $P,Q$ such that $X \oplus P$ is isomorphic to $Y \oplus Q$.

\begin{lemma}[Schanuel's Lemma]\label{lem:Schanuel}
If 
\[
   0\rightarrow K\rightarrow P\rightarrow X \rightarrow 0 \  \ \ \mbox{ and }
   \ \ \ 0\rightarrow K'\rightarrow P'\rightarrow X\rightarrow 0
\]
are two short exact sequences 
in an abelian category $\A$ with
$P,P'$ projective, then $K$ and $K'$ are stably isomorphic. In fact, $K\oplus P'\cong K'\oplus P$.
\end{lemma}

\begin{proof}
Consider the pullback of $P\stackrel{\phi}{\longrightarrow} X$
and $P'\stackrel{\psi}{\longrightarrow} X$, given by an object $W$ and morphisms
$W\stackrel{\alpha}{\longrightarrow} P$ and $W\stackrel{\beta}{\longrightarrow}P'$.
By definition, considering the zero morphisms from
$K$ to $X$ through $P$ and $P'$, there is a morphism from $K$ to $W$
that makes the corresponding diagram commute.
Similarly there is a morphism from $K'$ to $W$. 
By \cite[Theorem 6.2]{HS}, $K$ is the kernel of $\beta$
and $K'$ is the kernel of $\alpha$. Since 
$P\stackrel{\phi}{\longrightarrow} X$
and $P'\stackrel{\psi}{\longrightarrow} X$ 
are epimorphisms, the pullback diagram is
also a pushout diagram \cite[Exercise II.6.7]{HS} and
it follows that $\alpha$ and $\beta$ are epimorphisms.
To see this, note that if $Z$ is an object and $f,g: P\rightarrow Z$ are morphisms such that
$f\alpha = g\alpha$, then $(f-g)\alpha = 0$. Consider $f-g: P\rightarrow Z$
and the zero morphism $P'\rightarrow Z$. Since $X$ is a pushout,
there is a morphism $\nu: X\rightarrow Z$ such that $\nu\phi= f-g$.
If $f-g\neq 0$ then $\nu\neq 0$, but this contradicts the assumption
that $0 = \nu\psi$.  So $f=g$, implying that $\alpha$ is an epimorphism.
Similarly we see that $\beta$ is an epimorphism. 
Since $\alpha$ and $\beta$ are epimorphisms, and $P$ and $P'$
are projective, these morphisms split and we now have $W\cong K\oplus P'$
and $W\cong K'\oplus P$.
\end{proof}

The next lemma is a well-known characterization of split short exact sequences, but in the setting of quasi-Frobenius $k$-linear abelian categories. 

\begin{lemma}\label{lem:split}
For a short exact sequence
$$0 \to X \to Y \to Z \to 0$$
in a locally finite and quasi-Frobenius $k$-linear abelian category $\A$,
the following are equivalent:
\begin{itemize}
\item[(i)] The sequence splits;
\item[(ii)] $Y$ is isomorphic to $X \oplus Z$;
\item[(iii)] $Y$ is stably isomorphic to $X \oplus Z$.
\end{itemize}
\end{lemma}

\begin{proof}
The implications (i) $\Rightarrow$ (ii) $\Rightarrow$ (iii) are trivial, so suppose that (iii) holds, i.e.\ $Y \oplus P_1 \simeq X \oplus Z \oplus P_2$ for some projective objects $P_1$ and $P_2$. If either $X$ or $Z$ is projective, then the sequence splits, so suppose that this is not the case. By the Krull-Schmidt Theorem, we may decompose these two objects as $X \simeq X' \oplus P_X$ and $Z \simeq Z' \oplus P_Z$, where $P_X$ and $P_Z$ are projective, and $X'$ and $Z'$ have no projective direct summands. Then we split off $P_X$ and $P_Z$ from the short exact sequence and obtain a new one of the form
$$0 \to X' \to Y' \xrightarrow{\pi} Z' \to 0$$
with $Y \simeq Y' \oplus P_X \oplus P_Z$. Note that $Y'$ is stably isomorphic to $X' \oplus Z'$, and that this short exact sequence splits if and only if the original one does. 

Applying $\Hom_{\A}(Z',-)$ to this sequence, we obtain an exact sequence
$$0 \to \Hom_{\A}(Z',X') \to \Hom_{\A}(Z',Y') \xrightarrow{\pi_*} \Hom_{\A}(Z',Z') \to V \to 0$$
of finite dimensional $k$-vector spaces for some vector space $V$. From the isomorphism $Y \oplus P_1 \simeq X \oplus Z \oplus P_2$ and the three isomorphisms involving $X', Y'$ and $Z'$, we see that
$$Y' \oplus P_X \oplus P_Z \oplus P_1 \simeq X' \oplus Z' \oplus P_X \oplus P_Z \oplus P_2$$
and so by the Krull-Schmidt Theorem there is an isomorphism $Y' \simeq X' \oplus Z' \oplus P$ for some projective object $P$. Inserting $X' \oplus Z' \oplus P$ for $Y'$ in the four-term exact sequence above, and taking the alternating sum of the dimensions, we obtain
$$0 = \dim_k \Hom_{\A}(Z',P) + \dim_k V ,$$
hence $V = 0$. The map $\pi_*$ is then surjective, so the short exact sequence with $X',Y'$ and $Z'$ splits.   
\end{proof}

\subsection*{Minimal resolutions} 
A {\em projective resolution} $P_{\bu}$ of $X$ in $\C$ is an exact sequence 
\[
  \cdots \rightarrow P_2\rightarrow P_1\rightarrow P_0
      \rightarrow X\rightarrow 0
\]
in $\C$ such that $P_i$ is projective for each $i$.
Let $\Omega_{P_{\bu}}(X)$ be the kernel of the morphism $P_0\rightarrow X$,
and write $\Omega^1_{P_{\bu}}(X) = \Omega_{P_{\bu}}(X)$.
Let $\Omega^n_{P_{\bu}}(X) = \Omega_{P_{\bu}} (\Omega^{n-1}_{P_{\bu}}(X))$ for each $n> 1$,
where we view the morphism $P_{n-1}\rightarrow P_{n-2}$ as factoring
through $\Omega^{n-1}_{P_{\bu}}(X)$.
The objects $\Omega^n_{P_{\bu}}(X)$ depend on the projective resolution $P_{\bu}$, and are therefore not invariants of $X$. However, if $Q_{\bu}$ is another projective resolution of $X$, then $\Omega^n_{P_{\bu}}(X)$ and $\Omega^n_{Q_{\bu}}(X)$ are stably isomorphic for all $n$, by Schanuel's Lemma (Lemma~\ref{lem:Schanuel}).

A projective resolution $P_{\bu}$ of $X$ in $\C$ 
is {\em minimal} if $P_0=P(X)$ is a projective cover of $X$ and for each
$n\geq 1$, $P_n=P(\Omega^n_{P_{\bu}}(X))$ is a projective cover of $\Omega^n_{P_{\bu}}(X)$
(see, e.g.,~\cite[Section 7.9]{BD}).
Minimal resolutions exist and are unique up to isomorphism,
as a consequence of existence and uniqueness of
projective covers. For such a resolution $P_{\bu}$, we write $\Omega_{\C}^n(X)$ instead of $\Omega^n_{P_{\bu}}(X)$, since these objects are unique up to isomorphism and depend only on $X$.

Note that if we take any projective resolution of the unit object $\unit$, and tensor it with an object $X$, then the result is a projective resolution of $X$. In particular, the objects $\Omega_{\C}^n( \unit ) \otimes X$ and $\Omega_{\C}^n(X)$ are stably isomorphic. Also note that the existence of left and right duals implies that we may ``dualize'' everything we have done so far. Thus every object in a finite tensor category admits a minimal injective resolution, which is unique up to isomorphism, and we define $\Omega_{\C}^{-n}(X)$ using the cokernels in such a resolution.

We define $\Ext_{\C}^n(X,Y)$ for objects $X,Y$ just as we do in any abelian category with enough projective objects, namely by using any projective resolution of $X$. 
More specifically, for any two objects $X,Y$ of $\C$ and any nonnegative integer $n$, we define
\[
  \Ext^n_{\C} (X,Y) = \Coh^n (\Hom_{\C} (P_{\bu} , Y)) = \Ker d_{n+1}^*/\Ima
    d_n^* , 
\]
where $P_{\bu}$ is a projective resolution of $X$ 
with differentials $d_i:P_i\rightarrow P_{i-1}$, 
$d_i^*(f) = fd_i$ for all $i>0$ and $f\in\Hom_{\C}(P_{i-1}, Y)$, 
and $d_0^* = 0$.

\begin{lemma}\label{lem:Ext-Hom}
Let $P_{\bu}$ be a minimal projective resolution of an object
$X$ in $\C$, and let $X_i$ be a simple object of $\C$.
Then for all $n\geq 1$, 
\[
  \Ext^n_{\C}(X, X_i) \cong \Hom_{\C} (P_n,X_i)
  \cong \Hom_{\C}(\Omega^n_{\C}(X), X_i) . 
\]
\end{lemma}

\begin{proof}
We will show that the differentials on $\Hom_{\C}(P_{\bu}, X_i)$
are all zero maps.
The first isomorphism will then follow immediately.

Let $f\in \Hom_{\C}(P_n,X_i)$. 
If $f$ is nonzero, then it is an epimorphism since $X_i$ is simple.
Let $p:P(X_i)\rightarrow X_i$ be a projective cover, so that 
there is an epimorphism $g: P_n\rightarrow P(X_i)$ such that $pg=f$.
It follows that $P_n\cong P(X_i)\oplus Q_n$ for some projective object
$Q_n$, and under this isomorphism, $g$ may be viewed as 
the corresponding canonical projection onto $P(X_i)$.
Now if $fd_{n+1}\neq 0$, then $fd_{n+1}$ is an epimorphism,
and so there is an epimorphism $h:P_{n+1}\rightarrow P(X_i)$
such that $ph = fd_{n+1}$. 
Again, $h$ splits and 
$P_{n+1}\cong P(X_i)\oplus Q_{n+1}$ for a projective object $Q_{n+1}$, and
$h$ may be viewed as the corresponding canonical projection onto $P(X_i)$.

Now let $\hat{d}_{n+1} : P(X_i)\rightarrow P(X_i)$ denote the
following composition of morphisms: 
canonical inclusion of $P(X_i)$ into $P(X_i)\oplus Q_{n+1}$, then
isomorphism to $P_{n+1}$, then $d_{n+1}$, then isomorphism to $P(X_i)\oplus Q_n$,
then canonical projection onto $P(X_i)$.
Then $p\hat{d}_{n+1}= f d_{n+1}\mid_{P(X_i)} \neq 0$ by construction. 
Since $\dim_k\Hom_{\C}(P(X_i),X_i)=1$ by equation~(\ref{eqn:dim-Hom}), 
$p\hat{d}_{n+1} = \alpha p$ for some nonzero scalar $\alpha$.
Replacing $\hat{d}_{n+1}$ by $\alpha^{-1} \hat{d}_{n+1}$, we may assume that
$p\hat{d}_{n+1} = p$. 
We claim that this forces $\hat{d}_{n+1}$ to be an isomorphism
since $\dim_k\Hom_{\C}(P(X_i), X_j) = \delta_{i,j}$ by equation~(\ref{eqn:dim-Hom}).
To see this, note that 
if $\hat{d}_{n+1}$ were not an isomorphism, then $\Ima (\hat{d}_{n+1})$ would
be a subobject of $P(X_i)$, and necessarily a subobject of $Y_{m-1}$ in the
Jordan-H\"older series
\[
   0 = Y_0\subseteq \cdots \subseteq Y_{m-1}\subseteq Y_m = P(X_i) 
\]
with $Y_m/Y_{m-1} \cong X_i$. But then $p \hat{d}_{n+1}=0$,
a contradiction. 
Therefore $\hat{d}_{n+1}$ is an isomorphism.
However, this contradicts minimality of the projective resolution $P_{\bu}$ in the following way. 
By definition, 
$P_n$ is a projective cover of $\Omega^n_{\C}(X)$.
The image of $P(X_i)$ in $\Omega^n_{\C}(X)$ under the projective cover
morphism from $P_n\cong P(X_i)\oplus Q_n$ to $\Omega^n_{\C}(X)$
cannot be zero, as this would contradict the definition of projective cover.
To see this, map $Q_n$ to $\Omega^n_{\C}(X)$ via the canonical inclusion 
into $P_n\cong P(X_i)\oplus Q_n$ followed by
the epimorphism $P_n\rightarrow \Omega^n_{\C}(X)$.
This composite morphism is an epimorphism, since $P(X_i)$
is in the kernel of $P_n\rightarrow \Omega^n_{\C}(X)$.
However the length of $P_n$ is greater than that of $Q_n$,
so there can be no epimorphism from $Q_n$ to $P_n$,
contradicting the assumption that $P_n$ is the projective cover
of $\Omega^n_{\C}(X)$.
On the other hand, the image of $P(X_i)$ in $\Omega^n_{\C}(X)$ must be zero
since $\hat{d}_{n+1}$ is an isomorphism and $\Ima (d_{n+1})
\subseteq \Ker (d_n)$. This is a contradiction.
Therefore $fd_{n+1}=0$. 

The second isomorphism in the statement is induced by the epimorphism $P_n
\rightarrow \Omega^n_{\C}(X)$, since $fd_{n+1} =0$. 
\end{proof}

We note that for any two objects $X , Y$ of $\C$ and $n\geq 1$,
$\Ext^n_{\C}(X,Y)$ may be identified with 
equivalence classes of $n$-extensions of $Y$ by $X$~\cite{Oort}.
There are also long exact Ext sequences associated to
short exact sequences of objects in $\C$~\cite{Murfet}, and dimension shifting with respect to any projective resolution of $X$ works as one would expect.
We will use these facts about $\Ext^n_{\C}(X,Y)$ in the sequel.

\subsection*{Frobenius-Perron dimensions}\label{subsection: fpdim}

A very useful invariant and tool in the theory of finite tensor categories is the notion of Frobenius-Perron dimension.
Here we recall the definition and some of its 
properties that we will use.

As before, 
let $\C$ be a finite tensor category with (isomorphism classes of) simple objects $X_1, \ldots, X_r$. 
For each object $X$ in $\C$, let $N_X$ be the matrix of left multiplication by $X$, specifically 
\[ (N_X)_{ij} = ([X\otimes X_i:X_j])_{ij} ,\]
where $[X\otimes X_i:X_j]$ is the multiplicity of $X_j$ in a 
Jordan-H\"older series of the tensor product
object $X\ot X_i$ \cite[Section 1.5]{EGNO}.
The entries of this matrix are thus nonnegative integers.
The {\em Frobenius-Perron dimension} $\FPdim (X)$ of $X$ is the largest nonnegative real eigenvalue of the matrix $N_{X}$, which exists by the Frobenius-Perron Theorem~\cite[Theorem 3.2.1]{EGNO}.  
Moreover, $\FPdim (X_i) \geq 1$ for all 
$i = 1, \dots, r$~\cite[Proposition 3.3.4(2)]{EGNO}. 
Positivity characterizes the Frobenius-Perron dimension in the sense that $\FPdim$, extended by additivity to be a character of the Grothendieck ring of $\C$, is the unique such character that maps simple objects to positive real numbers~\cite[Proposition 3.3.6(3)]{EGNO}.
It follows that $\FPdim (X) >0$ for each nonzero object $X$ of $\C$,
since $\FPdim(X) = \sum_{i=1}^r a_i \FPdim(X_i)$
if $a_i=[X:X_i]$ for each $i$.

\section{Support varieties}\label{supp-var}

%%%%%%%%%%%%%%%%%%%%%%%%%%%%%%%%%%%%%%%

Here we adapt to finite tensor categories 
some of the definitions and results given in \cite{Feldvoss-W,Ostrik}
on support varieties for modules of finite dimensional Hopf algebras.
See also~\cite{BKSS} for tensor triangulated categories.  
These ideas originated in the theory of support varieties
for representations of finite groups 
(see Carlson~\cite{Carlson1}, Quillen~\cite{Quillen}, or the book by
Benson~\cite{Benson2}). 

%%%%%%%%
%%%%%%%
%%%%%%%

Support varieties for objects in $\C$ are defined in terms of cohomology. Given two objects $X,Y$, we denote the graded $k$-vector space $\oplus_{n=0}^{\infty} \Ext_{\C}^n(X,Y)$ by $\Ext_{\C}^*(X,Y)$. The Yoneda product turns $\Ext_{\C}^*(X,X)$ into a graded $k$-algebra, and $\Ext_{\C}^*(X,Y)$ into a graded left $\Ext_{\C}^*(Y,Y)$-module and a graded right $\Ext_{\C}^*(X,X)$-module. We denote the cohomology algebra $\Ext_{\C}^*( \unit, \unit )$ of the unit object $\unit$ by $\Coh^*( \C )$; this is the \emph{cohomology ring} of $\C$, and by \cite[Theorem 1.7]{SA} it is graded-commutative. Note that $\Coh^0(\C) = k$ since the unit object is simple, and k is algebraically closed.

The exact functor $- \otimes X$ induces a homomorphism
$$\Coh^*( \C ) \xrightarrow{\varphi_X} \Ext_{\C}^*(X,X)$$
of graded $k$-algebras, hence $\Ext_{\C}^*(X,Y)$ becomes a left $\Coh^*( \C )$-module via $\varphi_Y$, and a right $\Coh^*( \C )$-module via $\varphi_X$. By modifying the proof of \cite[Theorem 1.1]{SS04} to our setting, one can show that the left and right $\Coh^*( \C )$-module structures of $\Ext_{\C}^*(X,Y)$ coincide up to a sign. More precisely, if $\zeta \in \Coh^m( \C )$ and $\theta \in \Ext_{\C}^n(X,Y)$, then
$$\varphi_Y ( \zeta ) \circ \theta = (-1)^{mn} \theta \circ \varphi_X ( \zeta )$$
where the symbol $\circ$ denotes the Yoneda product. Consequently, when we view $\Ext_{\C}^*(X,Y)$ as a $\Coh^*( \C )$-module, it does not matter if we view it as a left or as a right module.

%%%%%%%%%
%%%%%%%%%
%%%%%%%%%

%%%%%%%%%%%%%%%%%%%%%%%%%%%%%%%%%
Since the cohomology ring of $\C$ is graded-commutative, its even part $\Coh^{2*}( \C ) = \oplus_{n=0}^{\infty} \Coh^{2n}( \C )$ is commutative. Furthermore, when the characteristic of the field $k$ is not two, then all the homogeneous elements in odd degrees are nilpotent, whereas when the characteristic is two, then the whole cohomology ring is commutative. We therefore make the following definition, of the commutative graded ring we shall use when we define support varieties.
\begin{definition}
For a finite tensor category $\C$ we define
$$\Ho ( \C ) = \left \{ 
\begin{array}{ll}
\Coh^*( \C ) & \text{if the characteristic of $k$ is two,} \\
\Coh^{2*}( \C ) & \text{if not.}
\end{array} 
\right.$$
\end{definition} 

Given objects $X,Y$ in $\C$, we denote by $I_{\C} (X,Y)$ the annihilator ideal of $\Ext_{\C}^*(X,Y)$  under the
action of $\Ho ( \C )$ described earlier. This is a homogeneous ideal, and whenever it is proper -- that is, when $\Ext_{\C}^*(X,Y)$ is nonzero -- it is contained in the unique maximal homogeneous ideal $\m_0 = \Coh^{+} (\C)$ of $\Ho ( \C )$. The \emph{support variety} of the pair $(X,Y)$ is now defined as
$$\VC(X,Y) \stackrel{\text{def}}{=} \{ \m_0 \} \cup \{ \m \in \Maxspec \Ho ( \C ) \mid I_{\C} (X,Y) \subseteq \m \}$$
where $\Maxspec \Ho ( \C )$ is the set of maximal ideals of $\Ho ( \C )$. Furthermore, the support variety of the single object $X$ is defined as $\VC(X) = \VC(X,X)$.  We also write $I_{\C}(X) = I_{\C}(X,X)$ and $\VC =\VC ( \unit )$.
Note that $\VC(X,Y) \subseteq \VC=\Maxspec \Ho(\C)$
for all objects $X,Y$. 
Note also that by definition, every support variety contains the point $\m_0$. If this were not part of the definition, then the variety of every pair of objects $X,Y$ with $\Ext_{\C}^*(X,Y) = 0$ would be empty; namely, in this case, the annihilator ideal $I_{\C} (X,Y)$ would necessarily be the whole cohomology ring $\Ho ( \C )$.
Finally, note that if $\Ext^n_{\C}(X,Y) = 0$ for $n \gg 0$,
then $\VC(X,Y) = \{ \M_0\}$. In particular, if $P$ is a projective object, then $\VC (P) = \{ \m_0 \}$.

The following lemma is useful in extending some of the classical properties of support varieties from other contexts (for example, group theory) to our more general setting. 

\begin{lemma}\label{lem:ideals}
\begin{itemize}
\item[(i)] For all objects $X$, $Y$ in $\C$, $I_{\C}(X) + I_{\C}(Y)\subseteq I_{\C}(X,Y)$.
\item[(ii)] For every exact sequence $0\to Y_1\to Y_2\to Y_3\to 0$ and all 
objects $W$ in $\C$,
 $I_{\C}(Y_j, W) \cdot I_{\C}(Y_l, W)\subseteq I_{\C}(Y_i,W)$ and $\VC(Y_i, W) \subseteq \VC(Y_j,W) \cup \VC(Y_l, W)$ whenever $\{i,j,l\} = \{1,2,3\}$. Similarly $I_{\C}(W,Y_j) \cdot I_{\C}(W,Y_l)\subseteq I_{\C}(W,Y_i)$ and $\VC(W,Y_i,) \subseteq \VC(W,Y_j) \cup \VC(W,Y_l)$ whenever $\{i,j,l\} = \{1,2,3\}$.
\end{itemize}
\end{lemma}

\begin{proof}
(i)
Recall from the beginning of this section that there are two equivalent definitions of the
action of $\Ho(\C)$ on $\Ext^*_{\C}(X, Y)$, one starting with $-\ot X$ and the other starting
with $ - \ot Y$, both followed by Yoneda composition of generalized extensions (and these actions coincide up to a sign). 
By definition, 
the action of $\Ho(\C)$ on $\Ext^*_{\C}(X, Y)$ factors through the action of $\Ho(\C)$ on $\Ext^*_{\C}(X, X)$ (and the same is true of $\Ext^*_{\C}(Y, Y)$). Thus $I_{\C}(X) + I_{\C}(Y)\subseteq I_{\C}(X,Y)$.

(ii)
Given an exact sequence $0\to Y_1\to Y_2\to Y_3\to 0$ and an object $W$ in $\C$,
we see from the long exact Ext sequence~\cite{Murfet}
\[
\cdots \rightarrow \Ext^n_{\C}(Y_3,W)\rightarrow \Ext^n_{\C}(Y_2,W) \rightarrow
  \Ext^n_{\C}(Y_1,W)\rightarrow \Ext^{n+1}_{\C}(Y_3,W)\rightarrow \cdots 
\]  
that there is a containment $I_{\C}(Y_j, W) \cdot I_{\C}(Y_l, W)\subseteq I_{\C}(Y_i,W)$.
Indeed, the product of two elements, one in $I_{\C}(Y_j,W)$ and one in $I_{\C}(Y_l,W)$,
necessarily annihilates both $\Ext^*_{\C}(Y_j,W)$ and 
$\Ext^*_{\C}(Y_l,W)$, and so in the long exact Ext sequence, 
the product acts as zero on the terms $\Ext^n_{\C}(Y_i,W)$.

Lastly, we will show the inclusion of the varieties.
If $\M\in \VC(Y_i, W)$, then from what we have just shown, there are inclusions $I_{\C}(Y_j, W) \cdot I_{\C}(Y_l, W)\subseteq I_{\C}(Y_i,W)\subseteq \M$. 
Since $\M$ is a maximal ideal, and therefore prime, it follows that $I_{\C}(Y_j, W) \subseteq \M$ or $I_{\C}(Y_l, W)\subseteq \M$. Thus $\M\in \VC(Y_j,W) \cup \VC(Y_l, W)$ as desired. The other half of the statement is proved similarly.
\end{proof}

%%%%%%%%%%%%%%%%%%%%%%%%%%%%%%%%

We now list some natural properties enjoyed by these support varieties, properties for which no finiteness condition is necessary. The following proposition holds just as in the finite group case and other more general settings; see, for example,~\cite[Section~5.7]{Benson2}. We include a proof for completeness. 

\begin{proposition}\label{prop:props}
Let $\C$ be a finite tensor category, % satisfying condition (fg), 
and let $X$, $Y$, $Y_1$, $Y_2$, and $Y_3$ be objects in $\C$. Then
\begin{itemize}
\item[(i)] $V_{\C}(X\oplus Y) = V_{\C}(X)\cup V_{\C}(Y)$.
\item[(ii)] $V_{\C}(X,Y)\subseteq V_{\C}(X)\cap V_{\C}(Y)$.
\item[(iii)] $V_{\C}(X) = \cup_{i=1}^r V_{\C}(X, X_i) 
    = \cup_{i=1}^r  V_{\C}(X_i,X)$, where $\{X_i \mid i=1,\ldots, r\}$ 
is a set of simple objects of $\mathcal C$, one from 
each isomorphism class.
\item[(iv)] If $\ 0\rightarrow Y_1\rightarrow Y_2\rightarrow Y_3\rightarrow 0$
is a short exact sequence, then 
\[
   \VC(Y_i)  \subseteq  \VC(Y_j) \cup \VC(Y_l) 
\]
whenever $\{i,j,l\} = \{1,2,3\}$. 
\item[(v)]$V_{\C}(X\ot Y) \subseteq V_{\C}(X)$, and if 
$\C$ is braided, then
$V_{\C}(X\ot Y)\subseteq V_{\C}(X)\cap V_{\C}(Y)$.
\item[(vi)] $V_{\C}(\Omega^1_{P_{\bu}} (X)) = V_{\C}(X)$ for any projective resolution $P_{\bu}$ of $X$.
\end{itemize}
\end{proposition}

\begin{proof}
(i)
First, we want to prove that $I_{\C}(X\oplus Y) = I_{\C}(X)\cap I_{\C}(Y)$ and then show that this implies that $V_{\C}(X\oplus Y) = V_{\C}(X)\cup V_{\C}(Y)$, as desired.

Since $\Ext^*_{\C}(X,X)\oplus \Ext^*_{\C}(Y,Y) \subseteq \Ext^*_{\C}(X\oplus Y, X\oplus Y) $, there is a containment of ideals, $I_{\C}(X\oplus Y) \subseteq I_{\C}(X)\cap I_{\C}(Y)$. For the other inclusion, recall also that $I_{\C}(X)\cap I_{\C}(Y)\subseteq I_{\C}(X) + I_{\C}(Y)\subseteq I_{\C}(X,Y)$, by Lemma~\ref{lem:ideals}(i). Then, since 
\[
   \Ext^*_{\C}(X\oplus Y, X\oplus Y) = \Ext^*_{\C}(X,X)\oplus \Ext^*_{\C}(X,Y)\oplus \Ext^*_{\C}(Y,X)\oplus \Ext^*_{\C}(Y,Y) ,
\]
it follows that $I_{\C}(X)\cap I_{\C}(Y)\subseteq  I_{\C}(X \oplus Y).$

We will check next that the equality $I_{\C}(X\oplus Y) = I_{\C}(X)\cap I_{\C}(Y)$ implies that $V_{\C}(X\oplus Y) = V_{\C}(X)\cup V_{\C}(Y)$, and we first show the inclusion $V_{\C}(X\oplus Y) \subseteq V_{\C}(X)\cup V_{\C}(Y)$. If $\M\in V_{\C}(X\oplus Y)$, then $I_{\C}(X\oplus Y) \subseteq \M$, and so $I_{\C}(X)\cap I_{\C}(Y)\subseteq \M$.
It is enough to show that if $I_{\C}(X)\not \subseteq \M$ then $I_{\C}(Y)\subseteq \M$. If $I_{\C}(X)\not \subseteq \M$, there exists $x\in I_{\C}(X)$ for which $x\notin \M$ and the ideal generated by $x$ and $\M$ generate the cohomology ring. 
Then $1 = ax + s$, with $a\in \Ho(\C)$, and  $s\in \M$. Now, if $y\in I_{\C}(Y)$, then $y = yax + ys$. 
Since $ I_{\C}(Y)$ and $ I_{\C}(X)$ are ideals,
$yax\in  I_{\C}(X) \cap  I_{\C}(Y)\subseteq \M$ and $ys\in \M$. So $ I_{\C}(Y)\subseteq \M$.

For the reverse inclusion $V_{\C}(X)\cup V_{\C}(Y) \subseteq V_{\C}(X\oplus Y)$, consider $\M\in V_{\C}(X)\cup V_{\C}(Y)$. Then $\M\in V_{\C}(X)$ or $\M\in V_{\C}(Y)$, that is $I_{\C}(X)\subseteq \M$ or $I_{\C}(Y)\subseteq \M$.
Since $I_{\C}(X\oplus Y) = I_{\C}(X)\cap I_{\C}(Y)\subseteq \M$, it follows that $\M \in V_{\C}(X\oplus Y)$.

(ii)
This is an immediate consequence of Lemma \ref{lem:ideals}(i). In fact, if $\M\in V_{\C}(X,Y)$, it follows that  $I_{\C}(X) + I_{\C}(Y)\subseteq I_{\C}(X,Y)\subseteq \M$. Then $\M\in V_{\C}(X)\cap  V_{\C}(Y)$.

(iii)
This follows from item (ii) in this proposition and Lemma \ref{lem:ideals}(ii). We will prove only one equality of the statement; the other can be shown in a similar way. 

By (ii), there is a containment $V_{\C}(X_i,X)\subseteq V_{\C}(X_i)\cap V_{\C}(X)\subseteq V_{\C}(X)$, for all $i=1, \ldots, r$. Then  $ \cup _i V_{\C}(X_i, X) \subseteq V_{\C}(X)$. For the other inclusion, recall that since our category $\C$ is a finite tensor category, the object $X$ has finite length. Associated to its Jordan-H{\"o}lder series are short exact sequences which, when we use Lemma \ref{lem:ideals}(ii), give $V_{\C}(X) \subseteq \cup _i V_{\C}(X_i, X)$.

(iv) 
By Lemma~\ref{lem:ideals}(ii), 
$V_{\C}(Y_i)\subseteq V_{\C}(Y_j,Y_i)\cup V_{\C}(Y_l, Y_i)$, 
and by part~(ii) of this proposition, that union is contained in
$V_{\C}(Y_j)\cup V_{\C}(Y_l)$.

(v)
By the definition of the action, we apply
$-\ot X$ and $-\ot X\ot Y$, respectively, followed in both cases by Yoneda composition of generalized extensions,  which implies that $I_{\C}(X) \subseteq I_{\C}(X\otimes Y)$. Consequently, $V_{\C}(X\ot Y) \subseteq V_{\C}(X)$.
Moreover, if $X\otimes Y \cong Y\otimes X$ (for example, when $\mathcal C$ is braided), we also conclude that $V_{\C}(X\ot Y) \subseteq V_{\C}(Y)$ via the same argument. Then, $V_{\C}(X\ot Y) \subseteq V_{\C}(X)\cap V_{\C}(Y)$ when 
$\C$ is braided. 

(vi) This follows from (iv) and the fact that $\VC (P) = \{ \m_0 \}$ for every projective object $P$ of $\C$.
\end{proof}

\begin{remark}\label{rem:one-sided}
In Proposition \ref{prop:props}(v), we see the first instance of the consequences of not assuming that our finite tensor category $\C$ is braided. Recall that when we defined the action of the cohomology ring $\coh^* ( \C )$ on $\Ext^*_{\C}(X,X)$, we used the ring homomorphism
$$\Coh^*( \C ) \xrightarrow{\varphi_X} \Ext_{\C}^*(X,X)$$
induced by the exact tensor product functor $- \ot X$. In doing so, we have made a \emph{choice}, namely that we tensor with $X$ on the right. We could instead choose the functor $X \ot -$, and had we done so, then the first inclusion in Proposition \ref{prop:props}(v) would be $V_{\C}(X\ot Y) \subseteq V_{\C}(Y)$.
\end{remark}

The properties we have just established hold in full generality, without any further assumptions on the cohomology of $\C$. We now state the finiteness condition mentioned earlier, and abbreviate it just \textbf{Fg}:
\begin{assumption}[\textbf{Fg}]
\sloppy The cohomology ring $\Coh^*( \C )$ is finitely generated, and $\Ext_{\C}^*(X,X)$ is a finitely generated $\Coh^*( \C )$-module for all objects $X \in \C$.
\end{assumption}
This was conjectured to hold for all finite tensor categories in \cite{EO}, and this conjecture is still open. A possibly weaker conjecture would be that the cohomology ring $\Coh^*( \C )$ is finitely generated, with no mention of $\Ext_{\C}^*(X,X)$. Note that $\Ext_{\C}^*(X,X)$ is a finitely generated $\Coh^*( \C )$-module for all objects $X$ if and only if $\Ext_{\C}^*(X,Y)$ is a finitely generated $\Coh^*( \C )$-module for all pairs of objects $X,Y$. This follows from the fact that every object in $\C$ has finite length, and that there are only finitely many isomorphism classes of simple objects in $\C$.

\begin{remark}
The finiteness condition \textbf{Fg} is stated in terms of the whole cohomology ring $\Coh^* (\C)$. However, it could just as well be stated in terms of $\Ho (\C)$. Namely, a finite tensor category $\C$ satisfies \textbf{Fg} if and only if $\Ho ( \C )$ is finitely generated, and $\Ext_{\C}^*(X,X)$ is a finitely generated $\Ho ( \C )$-module for all objects $X \in \C$. In the rest of the paper, we shall be using this fact without further mention.
\end{remark}

Without the finiteness condition \textbf{Fg}, the support varieties do not necessarily encode any important homological information; the properties listed above in Proposition \ref{prop:props} are just formal properties that do not tell much about the objects. However, when \textbf{Fg} holds, then the situation is very different as we will see in the rest of the paper.

Let us now recall some facts on varieties defined in terms of the maximal ideal spectrum of a commutative graded ring $R$ which is finitely generated over a field. For a homogeneous ideal $\az$ of $R$, we define $Z ( \az )$ as the set of maximal ideals of $R$ containing $\az$. Then $Z ( \az ) = Z ( \sqrt{\az} )$, where $\sqrt{\az}$ denotes the radical of $\az$. By \cite[Theorem 25]{Matsumura}, the radical of a proper ideal of $R$ is the intersection of all the maximal ideals containing it, hence if $\az$ and $\mathfrak{b}$ are proper homogeneous ideals with $Z ( \az ) \subseteq Z ( \mathfrak {b} )$, then $\sqrt{\mathfrak{b}} \subseteq \sqrt{\az}$. In particular, if $Z ( \az ) = Z ( \mathfrak {b} )$, then $\sqrt{\az} = \sqrt{\mathfrak{b}}$. Finally, we define the dimension of $Z ( \az )$ to be the Krull dimension of $R / \az$, or equivalently, the rate of growth $\gamma \left ( R / \az \right )$ of $R / \az$ (see the beginning of Section \ref{sec:complexity}) as a graded vector space; see \cite[Theorem 5.4.6]{Benson2}. This is well defined, for if $Z ( \az ) = Z ( \mathfrak {b} )$, then $\sqrt{\az} = \sqrt{\mathfrak{b}}$, and the Krull dimension of $R/ \az$ is clearly the same as that of $R/ \sqrt{\az}$.

The support variety of an object $X$ in $\C$ is by definition the set $Z ( I_{\C}(X) )$ associated to the commutative graded ring $\Ho ( \C )$. The {\em dimension} of $V_{\C}(X)$, denoted $\dim (V_{\C}(X))$, is then defined to be the dimension of this set, that is, the Krull dimension of $\Ho ( \C ) / I_{\C}(X)$. Of course, without the finiteness condition \textbf{Fg} this might actually be infinite, or a finite integer that does not reveal any information about the object $X$. However, as we shall see in Section \ref{sec:complexity}, when \textbf{Fg} holds, then this is a finite integer that measures the ``size'' of the minimal projective resolution of $X$.

\section{Complexity}\label{sec:complexity}

In this section, we define the complexity of an object $X$ as the rate of 
growth, defined next, 
of a minimal projective resolution. We then show that it is equal to the 
dimension of the support variety $V_{\C}(X)$ when \textbf{Fg} holds.

Let $a_{\bu} = (a_0,a_1,a_2,\ldots)$ be a sequence of nonnegative real numbers $a_i$.
The {\em rate of growth} $\gamma( a_{\bu})$ is defined to be
the smallest nonnegative integer $c$ for which there exists a real number $b$
such that $a_n \leq bn^{c-1}$ for all positive 
integers~$n$.
If no such $c$ exists, we define $\gamma(a_{\bu})$ to be $\infty$.
We will be interested in the rates of growth of sequences
$\dim_k W_{\bu}$ and $\FPdim P_{\bu}$ where $W_{\bu}$
is an $\N$-graded vector space over $k$ and $P_{\bu}$ is
an $\N$-graded object of $\C$.

Let $X$ be an object in $\C$.
The {\em complexity} of $X$ is defined to be the rate of growth
of a minimal projective resolution
$P_{\bu}$ of $X$ as measured by the Frobenius-Perron dimension:
\[
    \cx_{\C}(X) \stackrel{\text{def}}{=} \gamma ( \FPdim(P_{\bu})) .
\]
It follows from the proof of Theorem~\ref{thm:cx-dim} below
that our definition of complexity is equivalent to~\cite[Definition 4.1]{BKSS}. It also follows from the proof 
that all objects in a finite tensor
category satisfying condition \textbf{Fg} have finite complexity,
since the dimensions of the support varieties are necessarily finite. 
The proof of the theorem is in the same spirit as that for 
modules for finite group algebras~\cite[Proposition~5.7.2]{Benson2}.
However, we use Frobenius-Perron dimensions of objects
in place of vector space dimensions,
and exploit a connection with vector space dimension of Hom spaces.

\begin{theorem}\label{thm:cx-dim}
Let $\C$ be a finite tensor category satisfying condition \emph{\textbf{Fg}}.
For every object $X$ of $\C$, 
\[
   \cx_{\C}(X) = \dim  V_{\C}(X) \le \dim \Ho ( \C ),
\]
where $\dim \Ho ( \C )$ is the Krull dimension of $\Ho ( \C )$.
\end{theorem} 

\begin{proof}
By assumption \textbf{Fg}, $\Ext^*_{\C}(X,X)$ is a finitely generated module
over $\Ho (\C)$, and since the annihilator of this action is $I_{\C}(X)$
by definition, $\Ext^*_{\C}(X,X)$ is a finitely generated module
over $\Ho (\C)/I_{\C}(X)$. 
By definition, $\dim V_{\C}(X)$ is the Krull dimension of 
$\Ho (\C)/I_{\C}(X)$, which in turn is equal to its rate of growth as a
graded vector space,
so that we have 
\[
   \dim V_{\C}(X) = \dim ( \Ho (\C)/I_{\C}(X))  
  = \gamma (\dim_k \Ho (\C)/I_{\C}(X)) .
\]
The latter is equal to $\gamma (\dim_k \Ext^*_{\C}(X,X))$, 
since $\Ext^*_{\C}(X,X)$
is finitely generated as a module over the quotient $\Ho (\C)/I_{\C}(X)$.
Thus we must show that $\cx_{\C}(X) =\gamma(\dim_k \Ext^*_{\C}(X,X))$,
that is, we must show that
\[
   \gamma ( \FPdim (P_{\bu})) = \gamma (\dim_k \Ext^{\bu}_{\C}(X,X)) ,
\]
where $P_{\bu}$ is a minimal projective resolution of $X$ in $\C$. 
We will do this by proving that each quantity above
is less than or equal to the other. 

By Lemma~\ref{lem:proj},
the multiplicity of the projective cover $P(X_i)$ of a simple
object $X_i$ in $\C$, as a direct summand of $P_n$,
is $\dim_k \Hom_{\C}(P_n,X_i)$.
By Lemma~\ref{lem:Ext-Hom}, since $P_{\bu}$ is a minimal resolution,
\[
   \dim _k \Hom_{\C}(P_n,X_i) =\dim_k \Ext^n_{\C}(X,X_i) .
\]
We thus have 
\[
    \FPdim (P_n) = \sum_{i} \FPdim (P(X_i))
   \cdot\dim_k \Ext^n_{\C}(X,X_i) ,
\]
and so it follows that 
\[
  \gamma(\FPdim(P_{\bu})) \leq \max _i  \{ \gamma ( \dim_k \Ext^*_{\C}(X,X_i)) \} .
\]
Now condition \textbf{Fg} implies that each $\Ext^*_{\C}(X,X_i)$ is a finitely generated
$\Ho (\C)$-module, and since this action factors through 
$\Ext^*_{\C}(X,X)$ by the definition of the action, 
each $\Ext^*_{\C}(X,X_i)$ is finitely
generated as a module over $\Ext^*_{\C}(X,X)$. Thus 
\[
   \gamma(\dim_k \Ext^*_{\C}(X,X_i))\leq \gamma(\dim_k \Ext^*_{\C}(X,X))
\]
for each simple object $X_i$ in $\C$. 
It now follows from the two inequalities above that
\[\gamma ( \FPdim (P_{\bu})) \leq \gamma (\dim_k \Ext^{\bu}_{\C}(X,X)).\]

It remains to prove the reverse inequality. Since $\Ext^n_{\C}(X,X)$ is a subquotient of the vector
space $\Hom_{\C}(P_n,X)$, we have 
$
   \dim_k \Ext^n_{\C}(X,X) \leq \dim_k \Hom_{\C}(P_n,X) ,
$
and so
\[
    \gamma( \dim_k \Ext^{*}_{\C}(X,X)) \leq
   \gamma ( \dim_k \Hom_{\C}(P_{\bu},X)) .
\]
We claim that $\gamma (\dim_k \Hom_{\C}(P_{\bu},X))
\leq \gamma (\FPdim(P_{\bu}))$.
To see this, by Lemma~\ref{lem:proj}, we may 
write $P_n = \oplus _ i a_{n,i}P(X_i)$ for some nonnegative
integers $a_{n,i}$, 
where the $X_i$ are the simple objects.
By additivity of Hom and equation~(\ref{eqn:dim-Hom}) from Section \ref{sec:prelim},
\[
    \dim_k \Hom_{\C}(P_n,X)    
    =\sum_i a_{n,i} \dim_k \Hom_{\C}(P(X_i),X) 
   = \sum_i a_{n,i} [X:X_i],
\]
where $[X:X_i]$ is the multiplicity of $X_i$ as a composition factor
of $X$.
In addition, 
\[
   \FPdim(P_n) = \sum _i a_{n,i} \FPdim(P(X_i)) .
\]
Comparing rates of growth,  we now must show that 
\[
   \gamma( \sum_i a_{\bu,i} [X:X_i]) \leq \gamma(\sum_ia_{\bu,i}
   \FPdim(P(X_i))) .
\]
To see that this is indeed the case, note that for each $i$, 
the quantities $[X:X_i]$ and $\FPdim(P(X_i))$ are
{\em fixed} real numbers.
Moreover, $\FPdim(P(X_i))$ is positive for all $i$,
as explained in Section~\ref{sec:prelim}.  Therefore, in each expression, the rate of growth depends on 
the integers $a_{\bu,i}$.
The expression on the left side only depends on those $a_{\bu,i}$ 
for which $[X:X_i]$ is nonzero, and the expression on the right side depends on
all $a_{\bu,i}$, since $\FPdim(P(X_i))$ is nonzero for all $i$. Consequently, the inequality above holds. 
It follows that 
$\gamma(\dim_k\Hom_{\C}(P_{\bu},X)) \leq \gamma( \FPdim (P_{\bu}))$
as claimed,
and therefore 
$\gamma(\dim_k\Ext^*_{\C}(X,X)) 
  \leq   \gamma(\FPdim_k(P_{\bu})) .$
This shows that
\[
  \gamma(\FPdim(P_{\bu})) = \gamma (\dim_k \Ext^*_{\C}(X,X)) ,
\]
as required.
\end{proof}

As a consequence of the theorem, we obtain an expected result 
on the variety of a projective object. 

\begin{corollary}\label{cor:zero}
Let $\C$ be a finite tensor category satisfying condition \emph{\textbf{Fg}}.
An object $X$ of $\C$ is projective if and only if $\dim V_{\C}(X)=0$.
\end{corollary}

\begin{proof}
If $X$ is projective, then $\cx_{\C}(X) =0$ and so $\dim V_{\C}(X)=0$
by Theorem~\ref{thm:cx-dim}.
Conversely, if $\dim V_{\C}(X) =0$, then $X$ has a projective
resolution of finite length, say $0\rightarrow P_n\rightarrow\cdots
\rightarrow P_0\rightarrow X\rightarrow 0$. 
By~\cite[Proposition 6.1.3]{EGNO}, projective objects are also injective,
and so the morphism $P_n\rightarrow P_{n-1}$ splits, and similarly
for the other morphisms in the resolution.
This implies that $X$ is a
direct summand of $P_0$, and so is projective.
\end{proof}

We illustrate the notion of complexity and the theorem with an example
of Benson and Etingof~\cite{BE}. 

\begin{example}
Let $\C = \C_3$, the category defined in~\cite[Subsection 5.2.3]{BE}.
In \emph{loc.\ cit}., it is shown that this finite tensor category has two simple objects, $\unit$ and $V$,
with $\FPdim(\unit) =1$ and $\FPdim(V) = \sqrt{2}$.
Their projective covers $P(\unit)$ and $P(V)$ have 
Frobenius-Perron dimensions $\FPdim(P(\unit)) = 3+\sqrt{2}$ and 
$\FPdim(P(V))= 2 + 2\sqrt{2}$. 
It is also shown in~\cite{BE} that 
\[
  \Ext^*_{\C_3}(\unit,\unit) \cong k[x,y,z]/(y^2+xz) \ \ \ 
  \text{ and } \ \ \ \Ext^*_{\C_3}(V,V)\cong k[u,v]/(u^2) 
\]
where $|x|=1$, $|y|=|u|=2$, $|z|=|v|=3$, and
the minimal projective resolution of $V$ is 
\[
  \cdots\rightarrow P(V)\rightarrow P(\unit) \rightarrow P(V)
  \rightarrow P(V)\rightarrow P(\unit)\rightarrow P(V)\rightarrow V
  \rightarrow 0 .
\]
It follows that $\cx_{\C_3}(V)=1$ and $\cx_{\C_3}(\unit)=2$.
See~\cite{BE} for more details on the structure of this category.
\end{example}

\section{Carlson's $L_{\zeta}$ objects}\label{sec:carlson}

In this section, to each homogeneous element $\zeta$ of
the cohomology ring $\coh^*(\C)$ of the finite tensor category $\C$,
we associate an object $L_{\zeta}$ of $\C$.
These objects are defined analogously to Carlson's 
$L_{\zeta}$ modules for finite groups
and to Koszul objects in triangulated categories (see, e.g.~\cite{BKSS}),
and have similar useful properties.
We give a somewhat different approach to that in~\cite{BKSS},
and include proofs for completeness.

Let $n >0$.
By Lemma~\ref{lem:Ext-Hom}, since the unit object $\unit$ is simple,  
\[
   \coh^n(\C) = \Ext^n_{\C}(\unit, \unit) \cong \Hom_{\C} (\Omega^n_{\C}(\unit),\unit).
\]
Let $\zeta\in\coh^n(\C)$, and identify $\zeta$ with a morphism
$\hat{\zeta}$ from $\Omega^n_{\C}(\unit)$ to $\unit$ under the
above isomorphism.
Let $L_{\zeta}$ be its kernel, so that $L_{\zeta}$ is defined
by a short exact sequence:
\begin{equation*}\label{eqn:ell-zeta}
   0\rightarrow L_{\zeta} \rightarrow \Omega^n_{\C}(\unit)
   \stackrel{\hat{\zeta}}{\longrightarrow} \unit \rightarrow 0 . \tag{$\dagger \dagger$}
\end{equation*}
We shall prove that for every object $X$, the support variety $\VC (L_{\zeta} \ot X )$ is contained in $\VC (X)\cap Z( \zeta )$, with equality when condition \textbf{Fg} holds. This result generalizes~\cite[Proposition~3]{Pevtsova-W}
and~\cite[Theorem~2.5]{Feldvoss-W}, and parallels~\cite[Proposition 3.6]{BKSS};
our proof is essentially that in~\cite{Pevtsova-W}. For the proof we give, we need the following elementary lemma.

\begin{lemma}\label{lem:localization}
Let $R$ be a positively graded commutative ring, and $M$ a positively graded $R$-module. Furthermore, let $\mathfrak{p}$ be a prime ideal of $R$ not containing $R_+$, and consider the graded submodule $M_{\ge n}$ of $M$ for some $n \ge 0$. Then $( M_{\ge n} )_{\mathfrak{p}} = M_{\mathfrak{p}}$.
\end{lemma}

\begin{proof}
If $n > 0$, take an element $m/s$ in $M_{\mathfrak{p}}$ with $m \in \oplus_{i=0}^{n-1}M_i$ and $s \in R \setminus \mathfrak{p}$. Since $\mathfrak{p}$ does not contain $R_+$, there is a homogeneous element $t \in R_+$ with $t \notin \mathfrak{p}$, and so $m/s = t^jm/t^js$ in $M_{\mathfrak{p}}$ for all $j \ge 0$. The element $t^nm$ belongs to $M_{\ge n}$, hence $m/s \in ( M_{\ge n} )_{\mathfrak{p}}$.
\end{proof}

Now we prove the main result of this section. In the proof, we use the fact that when \textbf{Fg} holds, then  
\[
    \m \in \VC (X,Y) \Longleftrightarrow
   \Ext^*_{\C}(X,Y)_{\m} \neq 0
\]
for all objects $X,Y$ and every maximal ideal $\m \in \Ho ( \C )$, where the subscript $\m$ denotes localization.

\begin{theorem}\label{thm:ell-zeta}
Let $\C$ be a finite tensor category, $X$ an object in $\C$, and $\zeta\in \Ho (\C)$ a nonzero homogeneous element of positive degree.
Then 
\[
  \VC (L_{\zeta} \ot X ) \subseteq \VC (X)\cap Z( \zeta ) ,
\]
and equality holds if $\C$ satisfies condition \emph{\textbf{Fg}}.
In particular, when this is the case, then $\VC (L_{\zeta}) = Z( \zeta )$.
\end{theorem}

\begin{proof}
We first show that $\VC ( L_{\zeta} \ot X)\subseteq
\VC (X)\cap Z( \zeta )$. Denote the degree of $\zeta$ by $n$, and consider the short exact sequence~(\ref{eqn:ell-zeta}) defining the object $L_{\zeta}$. When we apply $- \ot X$ to it, we obtain a short exact sequence
$$0 \to L_{\zeta} \ot X \to \Omega_{\C}^n(X) \oplus P \to X \to 0$$
where $P$ is some projective object. It follows from properties (iv) and (vi) of Proposition~\ref{prop:props} that $\VC ( L_{\zeta} \ot X)\subseteq \VC (X)$. Furthermore, using Proposition~\ref{prop:props}(v) directly, we see that $\VC ( L_{\zeta} \ot X)\subseteq \VC(L_{\zeta})$, and so $\VC ( L_{\zeta} \ot X)\subseteq \VC (X) \cap \VC(L_{\zeta})$. We therefore need to show that $\VC(L_{\zeta})\subseteq Z( \zeta )$, and to do this it suffices to show
that $\VC(L_{\zeta},X_i)\subseteq Z( \zeta )$ for
each simple object $X_i$, by Proposition~\ref{prop:props}(iii).

By applying $\Hom_{\C}(-,X_i)$ to the short exact sequence ~(\ref{eqn:ell-zeta}), we obtain an exact sequence
$$\Ext^*_{\C}( \unit, X_i) \xrightarrow{\cdot \zeta} \Ext^{\ge n}_{\C}( \unit, X_i) \xrightarrow{f} \Ext^*_{\C}(L_{\zeta}, X_i) \xrightarrow{g} \Ext^{\ge 1}_{\C}( \unit, X_i) \xrightarrow{\cdot \zeta} \Ext^{\ge n+1}_{\C}( \unit, X_i)$$
of graded $\Ho (\C)$-modules, where we have used Lemma~\ref{lem:Ext-Hom}. Now take an element $\mu \in \Ext^*_{\C}(L_{\zeta}, X_i)$. Since
$$0 = \zeta \cdot g( \mu ) = g ( \zeta \cdot \mu ),$$
the element $\zeta \cdot \mu$ must belong to the image of the map $f$, that is, $\zeta \cdot \mu = f ( \theta )$ for some $\theta \in \Ext^{\ge n}_{\C}( \unit, X_i)$. This gives
$$\zeta^2 \cdot \mu = \zeta \cdot f ( \theta ) = f ( \zeta \cdot \theta ) = 0,$$
showing that $\zeta^2$ belongs to the annihilator ideal $I_{\C}(L_{\zeta}, X_i)$ of $\Ext^*_{\C}(L_{\zeta}, X_i)$ in $\Ho (\C)$. Now if $\m \in \VC(L_{\zeta},X_i)$, then by definition $I_{\C}(L_{\zeta}, X_i) \subseteq \m$, hence $\zeta^2 \in \m$. As $\m$ is a prime ideal, $\zeta$ must belong to $\m$, and therefore $\m \in Z( \zeta )$. This shows that $\VC(L_{\zeta},X_i)\subseteq Z( \zeta )$, and so we have proved the inclusion $\VC ( L_{\zeta} \ot X)\subseteq
\VC (X)\cap Z( \zeta )$.

For the reverse inclusion, suppose that \textbf{Fg} holds for $\C$. Again, using Proposition~\ref{prop:props}(iii), it suffices to show that
\[
   \VC (X,X_i)\cap Z( \zeta ) \subseteq
   \VC ( L_{\zeta} \ot X , X_i) 
\]
for every simple object $X_i$. Let therefore $\m$ be a maximal ideal of $\Ho ( \C )$ with $\m \notin \VC (L_{\zeta} \ot X,X_i )$. In particular, $\m \neq \m_0$, since $\m_0 \in \VC (L_{\zeta} \ot X,X_i )$ by definition. Moreover, since \textbf{Fg} holds, the localization $\Ext^*_{\C}(L_{\zeta} \ot X, X_i)_{\m}$ is zero.

Now apply $\Hom_{\C}(-,X_i)$ to the short exact sequence from the beginning of the proof, and obtain an exact sequence 
$$\Sigma^{-1} \Ext^*_{\C}(L_{\zeta} \ot X, X_i) \to \Ext^*_{\C}( X, X_i) \xrightarrow{\cdot \zeta} \Ext^{\ge n}_{\C}( X, X_i) \to \Ext^*_{\C}(L_{\zeta} \ot X, X_i)$$
of graded $\Ho (\C)$-modules, where we have used Lemma~\ref{lem:Ext-Hom} again. Here $\Sigma^{-1} \Ext^*_{\C}(L_{\zeta} \ot X, X_i)$ denotes the graded $\Ho (\C)$-module shifted in degree $-1$. The sequence remains exact when we localize at $\m$, and so since $\Ext^*_{\C}(L_{\zeta} \ot X, X_i)_{\m} =0$, the multiplication map
$$\Ext^*_{\C}( X, X_i)_{\m} \xrightarrow{\cdot \zeta} \Ext^{\ge n}_{\C}( X, X_i)_{\m}$$
is an isomorphism. It follows from Lemma \ref{lem:localization} that $\Ext^*_{\C}( X, X_i)_{\m} = \zeta \Ext^*_{\C}( X, X_i)_{\m}$. If $\m \in Z( \zeta )$, then this last equality implies that $\Ext^*_{\C}( X, X_i)_{\m} =0$ by Nakayama's lemma, since $\Ext^*_{\C}( X, X_i)_{\m} $ is a finitely generated $\Ho (\C)_{\m}$-module. But then $\m$ does not contain the annihilator ideal $I_{\C}(X,X_i)$ of $\Ext^*_{\C}( X, X_i)$ in $\Ho (\C)$, that is, $\m$ is not contained in $\VC (X,X_i)$. This shows that $\VC (X,X_i)\cap Z( \zeta ) \subseteq  \VC ( L_{\zeta} \ot X , X_i)$, and so we have proved the inclusion $\VC (X)\cap Z( \zeta ) \subseteq \VC ( L_{\zeta} \ot X)$.

Finally, note that the last statement of the theorem follows from the equality $\VC ( L_{\zeta} \ot X) =
\VC (X)\cap Z( \zeta )$, by setting $X = \unit$.
\end{proof}

We obtain as a consequence the following realization result.
Recall that a conical variety is by definition the
zero set of an ideal generated by homogeneous elements.

\begin{corollary}\label{cor:realization}
Let $\C$ be a finite tensor category satisfying condition \emph{\textbf{Fg}}, and $V$ any nonempty conical subvariety of $\VC$.
Then $V = \VC(X)$ for some object $X$ of $\C$.
\end{corollary}

\begin{proof}
By definition, $V = Z(I)$ for some homogeneous proper ideal $I$ of $\Ho (\C)$. Since $\Ho (\C)$ is Noetherian, this ideal is finitely generated, and so $I = \langle \zeta_1,\ldots, \zeta_t \rangle$ for some homogeneous elements $\zeta_1,\ldots,\zeta_t$ of positive degrees.
Let $X = L_{\zeta_1}\ot \cdots\ot L_{\zeta_t}$.
By Theorem~\ref{thm:ell-zeta}, $\VC(X) = Z(I) = V$.
\end{proof}

The following corollary shows that for every integer $c$ with $0 \le c \le \dim \Ho (\C)$, there exists an object $X \in \C$ of complexity $c$. Note that by Theorem~\ref{thm:cx-dim}, the complexity of every object in $\C$ is at most $\dim \Ho (\C)$. Hence every possible ``allowed'' complexity is realized by some object.

\begin{corollary}\label{cor:allcomplexities}
Let $\C$ be a finite tensor category satisfying condition \emph{\textbf{Fg}}, and $c$ an integer with $0 \le c \le \dim \Ho (\C)$, where $\dim \Ho (\C)$ denotes the Krull dimension of $\Ho (\C)$. Then there exists an object $X \in \C$ with $\cx_{\C}(X) = \dim \VC (X) = c$.
\end{corollary}

We end this section with a couple of general results involving the objects $L_{\zeta}$, results that do not require the \textbf{Fg} condition. The first one gives a necessary and sufficient condition for a homogeneous element in the cohomology ring $\Coh^*(\C)$ to annihilate the cohomology ring of an object. We will use this in the proof of the main theorem of Section~\ref{sec:indecomposable}. We do not actually need the ``sufficient'' part of this result; however, we include it for completeness and for possible future reference. 

\begin{proposition}\label{prop:zero}
Let $\C$ be a finite tensor category, $X \in \C$ an object, and $\zeta$ a nonzero element in $\Coh^n ( \C )$ for some $n \ge 1$. Then $\varphi_X ( \zeta ) = 0$ in $\Ext_{\C}^*(X,X)$ if and only if $\Omega_{\C}^{-1} ( L_{\zeta} ) \otimes X$ is stably isomorphic to $X \oplus \Omega_{\C}^{n-1}(X)$.
\end{proposition}

\begin{proof}
Consider the minimal projective resolution
$$\cdots \to P_2 \to P_1 \to P_0 \to \unit \to 0$$
of the unit object, and represent the element $\zeta$ by an epimorphism $\hat{\zeta} \colon \Omega_{\C}^n( \unit ) \to \unit$. Since the category $\C$ is abelian, we may take the pushout of $\hat{\zeta}$ with the monomorphism $\Omega_{\C}^n( \unit ) \to P_{n-1}$, and obtain a commutative diagram
$$\xymatrix{
& 0 \ar[d] & 0 \ar[d] \\
& L_{\zeta} \ar[d] \ar@{=}[r] & L_{\zeta} \ar[d] \\
0 \ar[r] & \Omega_{\C}^n( \unit ) \ar[d] ^{\hat{\zeta}} \ar[r] & P_{n-1} \ar[d] \ar[r] & \Omega_{\C}^{n-1}( \unit ) \ar@{=}[d] \ar[r] & 0 \\
0 \ar[r] & \unit \ar[d] \ar[r] & K_{\zeta} \ar[d] \ar[r] & \Omega_{\C}^{n-1}( \unit ) \ar[r] & 0 \\
& 0 & 0 }$$
with exact rows and columns. The bottom row corresponds to the element $\zeta$ under the dimension shift isomorphism $\Ext_{\C}^n( \unit, \unit ) \simeq \Ext_{\C}^1( \Omega_{\C}^{n-1}( \unit ), \unit )$, and the exactness of the second column shows that the object $K_{\zeta}$ is isomorphic to $\Omega_{\C}^{-1}( L_{\zeta} ) \oplus P$ for some projective object $P$. The image $\varphi_X ( \zeta )$ of $\zeta$ in $\Ext_{\C}^*(X,X)$ is therefore represented by the short exact sequence
$$0 \to X \to \left ( \Omega_{\C}^{-1}( L_{\zeta} ) \otimes X \right ) \oplus \left ( P \otimes X \right ) \to \Omega_{\C}^{n-1}( \unit ) \otimes X \to 0$$
where $\Omega_{\C}^{n-1}( \unit ) \otimes X \simeq \Omega_{\C}^{n-1}( X ) \oplus Q$ for some projective object $Q$, and where we use again a dimension shift isomorphism $\Ext_{\C}^n(X,X ) \simeq \Ext_{\C}^1( \Omega_{\C}^{n-1}( X ),X )$. Now, the image $\varphi_X ( \zeta )$ is zero in $\Ext_{\C}^*(X,X)$ if and only if this short exact sequence splits, which by Lemma \ref{lem:split} happens if and only if $\left ( \Omega_{\C}^{-1}( L_{\zeta} ) \otimes X \right ) \oplus \left ( P \otimes X \right )$ is stably isomorphic to $X \oplus \Omega_{\C}^{n-1}( X ) \oplus Q$. As $P \otimes X$ and $Q$ are projective objects, this is equivalent to $\Omega_{\C}^{-1}( L_{\zeta} ) \otimes X$ being stably isomorphic to $X \oplus \Omega_{\C}^{n-1}( X )$.
\end{proof}

The final result in this section relates the objects $L_{\zeta_1}, L_{\zeta_2}$ and $L_{\zeta_1 \zeta_2}$ for two homogeneous elements $\zeta_1, \zeta_2 \in \Coh^* ( \C )$.

\begin{proposition}\label{prop:product}
Let $\C$ be a finite tensor category, and $\zeta_1, \zeta_2$ two homogeneous elements in $\Coh^* ( \C )$. Then there is a short exact sequence
$$0 \to \Omega_{\C}^{| \zeta_1 |}( L_{\zeta_2} ) \to L_{\zeta_1 \zeta_2} \oplus P \to L_{\zeta_1} \to 0$$
for some projective object $P$.
\end{proposition}

\begin{proof}
The proof from the group cohomology case carries over; see \cite[Lemma 5.9.3]{Benson2}.
\end{proof}

%%%%%%%%%%%%%%%%%%%%%%%%%%%%%%%%%%%%%%%%%%%%%%%%%%

\section{Support varieties of indecomposable objects}\label{sec:indecomposable}

In this section, we prove that when the finiteness condition \textbf{Fg} holds, then the support variety of an indecomposable object is connected. We start with the following result, which allows us to, in a sense, reduce the complexity of an object.  

\begin{proposition}\label{prop:reducing}
If $\C$ is a finite tensor category satisfying \emph{\textbf{Fg}}, and $X$ is an object with $\dim \VC(X) \ge 1$, then there exists a short exact sequence
$$0 \to X \to K \to \Omega_{\C}^n(X) \to 0$$
for some $n \ge 0$, with $\dim \VC(K) = \dim \VC(X)-1$.
\end{proposition}

\begin{proof}
Consider the annihilator ideal $I = I_{\C}(X)$ of $\Ext_{\C}^*(X,X)$ in $\Ho ( \C )$. By definition, the dimension of $\VC(X)$ is the Krull dimension of $\Ho ( \C ) / I$, which by \cite[Theorem 5.4.6]{Benson2} equals its rate of growth $\gamma \left ( \Ho ( \C ) / I \right )$ as a graded $k$-vector space. By assumption, this is a positive integer.

\sloppy By \cite[Lemma 2.5]{BIKO} there exists a homogeneous element $\zeta \in \Ho ( \C )$, of positive degree, say $n$, with the property that multiplication
$$\left ( \Ho ( \C ) / I \right )_i \xrightarrow{\cdot \zeta} \left ( \Ho ( \C ) / I \right )_{i+n}$$
is injective for $i \gg 0$. Choose an $n_0$ such that these multiplication maps are injective for $i \ge n_0$, and consider the exact sequence
$$0 \to \bigoplus_{i=n_0}^{\infty} \left ( \Ho ( \C ) / I \right )_i \xrightarrow{\cdot \zeta} \bigoplus_{i=n_0}^{\infty} \left ( \Ho ( \C ) / I \right )_i \to \bigoplus_{i=n_0}^{\infty} \left ( \Ho ( \C ) / ( I, \zeta ) \right )_i \to 0$$
of graded $k$-vector spaces. By the Hilbert-Serre Theorem, the Poincar{\'e} series of $\oplus_{i=n_0}^{\infty} \left ( \Ho ( \C ) / I \right )_i$ is a rational function of the form $f(t) / \prod (1 - t^{m_i})$; see \cite[Proposition 5.3.1]{Benson2}. Moreover, by \cite[Proposition 5.3.2]{Benson2}, the rate of growth of $\oplus_{i=n_0}^{\infty} \left ( \Ho ( \C ) / I \right )_i$ is the same as the order of the pole at $t=1$ of this rational function. Similarly, the Poincar{\'e} series of $\oplus_{i=n_0}^{\infty} \left ( \Ho ( \C ) / ( I, \zeta ) \right )_i$ is of the form $g(t) / \prod (1 - t^{m_i})$, and its rate of growth is the order of the pole at $t=1$ of this rational function. 

Since $\zeta$ is regular on $\oplus_{i=n_0}^{\infty} \left ( \Ho ( \C ) / I \right )_i$, it follows from \cite[Proposition 11.3]{AtiyahMacdonald} that the rate of growth of $\oplus_{i=n_0}^{\infty} \left ( \Ho ( \C ) / (I, \zeta ) \right )_i$ is one less than that of $\oplus_{i=n_0}^{\infty} \left ( \Ho ( \C ) / I \right )_i$. The rate of growth of a graded vector space does not change when we discard finitely many homogeneous subspaces, hence $\gamma \left ( \Ho ( \C ) / ( I, \zeta ) \right ) = \gamma \left ( \Ho ( \C ) / I \right ) -1$. Now consider the commutative diagram from the proof of Proposition \ref{prop:zero}. Applying $- \otimes X$ to this diagram, we obtain a short exact sequence
$$0 \to X \to K_{\zeta} \otimes X \to \Omega_{\C}^{n-1}( \unit ) \otimes X \to 0$$
with $\VC ( K_{\zeta} \otimes X ) = \VC ( L_{\zeta} \otimes X )$ in light of the second column and the fact that $P_{n-1} \otimes X$ is a projective object; see Proposition~\ref{prop:props}(iv). Therefore, by Theorem~\ref{thm:ell-zeta}, there are equalities
$$\VC ( K_{\zeta} \otimes X ) = \VC ( X ) \cap Z ( \zeta ) = Z ( I ) \cap Z ( \zeta ) = Z ( I, \zeta  ) . $$
The dimension of $\VC ( K_{\zeta} \otimes X )$ is then the Krull dimension of $\Ho ( \C ) / ( I, \zeta )$, which is one less than that of $\Ho ( \C ) / I$ by the above. This shows that $\dim \VC ( K_{\zeta} \otimes X ) = \dim \VC ( X ) -1$. Finally, note that in the short exact sequence above, the object $\Omega_{\C}^{n-1}( \unit ) \otimes X$ is isomorphic to $\Omega_{\C}^{n-1}( X ) \oplus P$ for some projective object $P$. Splitting this $P$ off from the sequence, we obtain an object $K$ and a short exact sequence
$$0 \to X \to K \to \Omega_{\C}^{n-1}(X) \to 0$$
with $\VC(K) = V_{\C} ( K_{\zeta} \otimes X )$.
\end{proof}

In the following result, we characterize when $\VC(X,Y)$ is trivial, that is, zero-dimensional.
 
\begin{proposition}\label{prop:vanishing}
If $\C$ is a finite tensor category satisfying \emph{\textbf{Fg}}, then the following are equivalent for all objects $X,Y$:
\begin{itemize}
\item[(i)] $\dim \VC(X,Y) =0$;
\item[(ii)] $\Ext_{\C}^n(X,Y) =0$ for $n \gg 0$;
\item[(iii)] $\Ext_{\C}^n(X,Y) =0$ for $n \ge 1$.
\end{itemize}
\end{proposition}

\begin{proof}
If $\Ext_{\C}^n(X,Y) =0$ for $n \gg 0$, then for large $i$ the homogeneous subspace $I_{\C}(X,Y)_i$ of $I_{\C}(X,Y)$ equals $\left ( \Ho ( \C ) \right )_i$. Then $\left ( \Ho ( \C ) / I_{\C}(X,Y) \right )_i =0$ for $i \gg 0$, hence $\dim \VC(X,Y) = \gamma \left ( \Ho ( \C ) / I_{\C}(X,Y) \right ) =0$. Conversely, if the rate of growth of $\Ho ( \C ) / I_{\C}(X,Y)$ is zero, then $\left ( \Ho ( \C ) / I_{\C}(X,Y) \right )_i = 0$ for $i \gg 0$. As $\Ext_{\C}^*(X,Y)$ is a finitely generated graded module over $\Ho ( \C ) / I_{\C}(X,Y)$, we conclude that $\Ext_{\C}^n(X,Y) =0$ for $n \gg 0$. This proves the equivalence of (i) and (ii).

\sloppy We now show by induction on the dimension of $\VC(X)$ that (ii) implies (iii). If $\dim \VC(X) =0$, then $X$ is a projective object by Corollary~\ref{cor:zero}, and so trivially $\Ext_{\C}^n(X,Y) =0$ for $n \ge 1$. If the dimension of $\VC(X)$ is nonzero, then choose, by Proposition \ref{prop:reducing}, a short exact sequence
$$0 \to X \to K \to \Omega_{\C}^t(X) \to 0$$
for some $t \ge 0$, with $\dim \VC(K) = \dim \VC(X)-1$. We obtain from this sequence a long exact sequence
$$\resizebox{.98\hsize}{!}{$\Ext_{\C}^{1+t}(X,Y) \to \Ext_{\C}^1(K,Y) \to \Ext_{\C}^1(X,Y) \to \Ext_{\C}^{2+t}(X,Y) \to \Ext_{\C}^2(K,Y) \to \cdots$}$$
in cohomology, where we have used dimension shift to replace $\Ext_{\C}^i( \Omega_{\C}^t(X),Y)$ by $\Ext_{\C}^{i+t}(X,Y)$. By assumption, the cohomology groups $\Ext_{\C}^n(X,Y)$ vanish for $n \gg 0$, and so from the long exact sequence we see that the same is true for the cohomology groups $\Ext_{\C}^n(K,Y)$. But then by induction $\Ext_{\C}^n(K,Y) =0$ for $n \ge 1$, implying that $\Ext_{\C}^n(X,Y)$ and $\Ext_{\C}^{n+t+1}(X,Y)$ are isomorphic for all $n \ge 1$. Since $\Ext_{\C}^n(X,Y)=0$ for $n \gg 0$, we conclude that $\Ext_{\C}^n(X,Y) =0$ for $n \ge 1$.
\end{proof}

We are now ready to prove the main result in this section. It shows that if the support variety of an object can be written as the union of two subvarieties having trivial intersection, then the object decomposes accordingly into a direct sum. The proof is an adaption of Benson's proof of \cite[Theorem 5.12.1]{Benson2}, based on Carlson's original proof from \cite{Carlson}.

\begin{theorem}\label{thm:main}
Let $\C$ be a finite tensor category satisfying \emph{\textbf{Fg}}, and $X$ an object in $\C$. Suppose that $\VC(X) = V_1 \cup V_2$, where $V_1$ and $V_2$ are conical subvarieties of $\VC(X)$ with $V_1 \cap V_2 = \{ \m_0 \}$. Then $X \simeq X_1 \oplus X_2$ for some objects $X_1$ and $X_2$ with $\VC(X_i) = V_i$.
\end{theorem}

\begin{proof}
The proof is by induction on the sum $\dim V_1 + \dim V_2$. If either $\dim V_1$ or $\dim V_2$ is zero, then we just take the corresponding $X_i$ to be the zero object, and the other to be $X$. We may therefore suppose that both $\dim V_1$ and $\dim V_2$ are nonzero, so that there exist proper homogeneous ideals $I_1$ and $I_2$ of $\Ho ( \C )$ with $V_i = Z ( I_i )$, and such that the Krull dimension of $\Ho ( \C ) / I_i$ is nonzero.

Choose a homogeneous element $\zeta \in \Ho ( \C )$, of positive degree, with the property that the Krull dimension of $\Ho ( \C ) / ( I_2, \zeta )$ is one less than that of $\Ho ( \C ) / I_2$; in the proof of Proposition \ref{prop:reducing} we showed that such an element exists. By assumption, there are equalities
$$Z ( I_1 + I_2 ) = Z ( I_1 ) \cap Z ( I_2 ) = V_1 \cap V_2 = \{ \m_0 \}$$
and so the radical of $ I_1 + I_2$ must equal $\m_0$. Therefore $\zeta^t \in  I_1 + I_2$ for some $t$, giving $\zeta^t = \zeta_1 + \theta$ for some homogeneous elements $\zeta_1 \in I_1$ and $\theta \in I_2$. The Krull dimensions of $\Ho ( \C ) / ( I_2, \zeta )$ and $\Ho ( \C ) / (I_2, \zeta^t )$ are clearly the same, hence
\begin{eqnarray*}
\dim \left ( \Ho ( \C ) / ( I_2, \zeta_1 ) \right ) & = & \dim \left ( \Ho ( \C ) / ( I_2, \zeta_1 + \theta ) \right ) \\
& = & \dim \left ( \Ho ( \C ) / ( I_2, \zeta^t ) \right ) \\
& = & \dim \left ( \Ho ( \C ) / I_2 \right ) - 1 .
\end{eqnarray*}
Similarly, we can find a homogeneous element $\zeta_2 \in I_2$, of positive degree, with the property that the Krull dimension of $\Ho ( \C ) / ( I_1, \zeta_2 )$ is one less than that of $\Ho ( \C ) / I_1$.  

Since $\zeta_i \in I_i$, there is an inclusion $V_i \subseteq Z ( \zeta_i )$. This gives
$$Z ( I_{\C}(X,X) ) = \VC(X) = V_1 \cup V_2 \subseteq Z ( \zeta_1 ) \cup Z ( \zeta_2 ) = Z ( \zeta_1 \zeta_2 )$$
and so $\zeta_1 \zeta_2$ belongs to $\sqrt{I_{\C}(X,X)}$. Again, the Krull dimensions of $\Ho ( \C ) / ( I_2, \zeta_1 )$ and $\Ho ( \C ) / ( I_1, \zeta_2 )$ remain the same when we replace $\zeta_1$ and $\zeta_2$ by powers, and so we may assume that $\zeta_1 \zeta_2 \in I_{\C}(X,X)$. Then by Proposition \ref{prop:zero}, the objects $\Omega_{\C}^{-1} ( L_{\zeta_1 \zeta_2} ) \otimes X$ and $X \oplus \Omega_{\C}^{n-1}(X)$ are stably isomorphic, where $n = | \zeta_1 \zeta_2 |$. Note that $\Omega_{\C}^{-1} ( L_{\zeta_1 \zeta_2} ) \otimes X$ is stably isomorphic to $\Omega_{\C}^{-1} ( L_{\zeta_1 \zeta_2} \otimes X )$, and so when we apply $\Omega_{\C}^1$, we see that $L_{\zeta_1 \zeta_2} \otimes X$ is stably isomorphic to $\Omega_{\C}^1(X) \oplus \Omega_{\C}^{n}(X)$. Now apply $- \otimes X$ to the short exact sequence in Proposition \ref{prop:product}. Using what we have just seen, we obtain a short exact sequence
$$0 \to \Omega_{\C}^{r}( L_{\zeta_2} ) \otimes X \to \Omega_{\C}^1(X) \oplus \Omega_{\C}^{n}(X) \oplus Q \to L_{\zeta_1} \otimes X \to 0$$
where $r$ is the degree of $\zeta_1$, and $Q$ is a projective object. 

Consider the end terms of this short exact sequence. The object $\Omega_{\C}^{r}( L_{\zeta_2} ) \otimes X$ is stably isomorphic to $\Omega_{\C}^{r}( L_{\zeta_2} \otimes X)$, and support varieties are invariant under syzygies. Therefore, since $V_i \subseteq Z( \zeta_i )$, we see from Theorem~\ref{thm:ell-zeta} that
$$\VC( \Omega_{\C}^{r}( L_{\zeta_2} ) \otimes X ) = Z( \zeta_2 ) \cap \VC(X) = Z( \zeta_2 ) \cap (V_1 \cup V_2) = ( Z( \zeta_2 ) \cap V_1 ) \cup V_2$$
and
$$\VC( L_{\zeta_1} \otimes X)  = Z( \zeta_1 ) \cap \VC(X) = Z( \zeta_1 ) \cap (V_1 \cup V_2) = V_1 \cup ( Z( \zeta_1 ) \cap V_2 ) . $$
Let us denote $Z( \zeta_2  ) \cap V_1$ by $V_1'$, and $Z( \zeta_1 ) \cap V_2$ by $V_2'$. Note that $Z( \zeta_2 ) \cap V_1 = Z ( \zeta_2 ) \cap Z ( I_1 ) = Z ( I_1, \zeta_2 )$, hence the dimension of the variety $V_1'$, that is, the Krull dimension of $\Ho ( \C ) / ( I_1, \zeta_2 )$, equals $\dim V_1 -1$. Similarly, $\dim V_2' = \dim V_2 -1$. To sum up: the support varieties of the two objects $ \Omega_{\C}^{r}( L_{\zeta_2} ) \otimes X$ and $L_{\zeta_1} \otimes X$ decompose as
$$\VC( \Omega_{\C}^{r}( L_{\zeta_2} ) \otimes X ) = V_1' \cup V_2$$
and 
$$\VC( L_{\zeta_1} \otimes X) = V_1 \cup V_2'$$
with both the sums $\dim V_1' + \dim V_2$ and $\dim V_1 + \dim V_2'$ equal to $\dim V_1 + \dim V_2 -1$. Moreover, from the construction of $V_1'$ and $V_2'$, it is clear that $V_1' \cap V_2 = \{ \m_0 \}$ and $V_1 \cap V_2' = \{ \m_0 \}$.

By induction, we can decompose the objects into direct sums $\Omega_{\C}^{r}( L_{\zeta_2} ) \otimes X \simeq Y_1 \oplus Y_2$ and $L_{\zeta_1} \otimes X \simeq Z_1 \oplus Z_2$, with $\VC(Y_1) = V_1'$ and $\VC(Y_2) = V_2$, and with $\VC(Z_1) = V_1$ and $\VC(Z_2) = V_2'$. Both the intersections $\VC(Y_1) \cap \VC(Z_2)$ and $\VC(Y_2) \cap \VC(Z_1)$ equal $\{ \m_0 \}$, and so it follows from Proposition~\ref{prop:props}(ii) and Proposition~\ref{prop:vanishing} that $\Ext_{\C}^1(Y_1,Z_2) =0$ and $\Ext_{\C}^1(Y_2,Z_1) =0$. Consequently, the short exact sequence above is isomorphic to the direct sum of two short exact sequences
$$0 \to Y_1 \to X_1' \to Z_1 \to 0$$
$$0 \to Y_2 \to X_2' \to Z_2 \to 0$$
for some objects $X_1'$ and $X_2'$. In particular, the object $\Omega_{\C}^1(X) \oplus \Omega_{\C}^{n}(X) \oplus Q$ is isomorphic to $X_1' \oplus X_2'$. Applying Proposition~\ref{prop:props} to these two short exact sequences, we see that $\VC(X_1') \subseteq V_1$ and $\VC(X_2') \subseteq V_2$. The Krull-Schmidt Theorem, the fact that support varieties are invariant under syzygies, and the fact that $\VC(X_1') \cap \VC(X_2') = \{ \m_0 \}$ now imply that the object $X$ must decompose as $X \simeq X_1 \oplus X_2$, with $\VC(X_i) = \VC(X_i') \subseteq V_i$. But $\VC(X) = V_1 \cup V_2$, and so $\VC(X_i)$ must equal $V_i$ for each $i$. This concludes the proof.
\end{proof}

By removing the origin, i.e.\ the unique homogeneous maximal ideal $\m_0$, the support varieties become projective varieties. From the theorem it is then clear that the projective support variety of an indecomposable object is connected.

\begin{corollary}\label{cor:connected}
In a finite tensor category satisfying \emph{\textbf{Fg}}, the projective support variety of an indecomposable object is connected.
\end{corollary}

%%%%%%%%%%%%%%%%%%%%%%%%%%%%%%%%%%%%%%%%%%%%%

%%%%%%%%%%%%%%%%%%%%%%%%%%%%%%%%%%%%%%%%%%%%%%%%

\end{document}